\documentclass[12pt,letterpaper]{article}
\usepackage[margin=1.0in]{geometry}
\usepackage[utf8]{inputenc}
\usepackage{amsmath,amsthm,amsfonts}
\usepackage{graphicx}
\usepackage{float}
\usepackage{hyperref}
\usepackage{euler}
\usepackage{comment}
\usepackage{pgfplots}
\usepackage{mathrsfs}
\usetikzlibrary{arrows}
\usepackage{subcaption}
\usepackage[
backend=biber,
style=alphabetic
]{biblatex}
\usepackage{subcaption}
\usepackage[labelformat=parens,labelsep=quad,skip=3pt]{caption}
\newtheorem{theorem}{Theorem}
\newtheorem{Lemma}{Lemma}
\newtheorem{remark}{Remark}
\newtheorem{note}{Note}

\newtheorem{proposition}{Proposition}

\newcommand{\overbar}[1]{\mkern 1.5mu\overline{\mkern-1.5mu#1\mkern-1.5mu}\mkern 1.5mu}
  
 \newcommand\newsubcap[1]{\phantomcaption%
       \caption*{\figurename~\thefigure(\thesubfigure): #1}}
 
\title{Countable modular groups of infinite type surfaces}
\author{Rogelio Ni\~no Hern\'andez}
\date{\today}

\newcommand{\Address}{
  \bigskip
  \footnotesize

   \textsc{Centro de Ciencias Matemáticas-UNAM,
     Antigua Carretera a Pátzcuaro \# 8701, Sin Nombre, Residencial San José de la Huerta, 58089 Morelia, Mich.}\par\nopagebreak
  \textit{E-mail address}:\texttt{rnino@matmor.unam.mx}
}

\begin{document}

\maketitle

\begin{abstract}
 We prove that every connected, orientable infinite type surface $S$ without boundary and finite genus has a Riemann surface structure such that its modular group of quasiconformal homeomorphisms is countable.     
\end{abstract}
\section{Introduction}

Consider a surface without boundary, connected and orientable denoted as $S$, and its topological mapping class group, defined as the quotient of the group of orientation-preserving homeomorphisms of $S$ by the subgroup of homeomorphisms isotopic to the identity ($Map(S) = Homeo^+(S)/Homeo_0(S)$). Consider also a Riemann surface structure $R$ on $S$ and the subgroup $Mod(R)$ of $Map(S)$ for which there exists a quasiconformal representative with respect to $R$. 

In the case where the surface $S$ is of finite topological type (meaning its fundamental group is finitely generated), both $Map(S)$ and $Mod(R)$ are always isomorphic and countable. However, when the surface $S$ is of infinite topological type (with a non-finitely generated fundamental group), the situation changes. In this context, the mapping class group $Map(S)$ is always uncountable. Furthermore, certain homeomorphisms cannot be realised as elements within $Mod(R)$, regardless of the chosen Riemann surface structure $R$. This follows from considering infinite products of Dehn twists by increasing powers (see \cite{Dehntwistsproducts}). 

In his work \cite{Matsuzaki}, Matsuzaki shows that for the surface $S = \mathbb{S}^2 \setminus C$, where $C$ represents the Cantor set, there exists a Riemann surface structure $R$, wherein the associated group $Mod(R)$ is countable.  Chandra, Patel and Vlamis \cite{Vlamis} ask "How does $Mod(R)$ sits in $Map(S)$?", we extend Matsuzaki's results to give a partial answer to this question. Our main result is the following. 

\begin{theorem}
\label{MainResult}
For every connected, orientable infinite type surface $S$ of finite genus there exists a Riemann surface structure $R$ such that $Mod(R)$ is countable.

\end{theorem}

The strategy we use follows closely that of Matsuzaki \cite{Matsuzaki}, which can be roughly summarised as follows:
using the notion of a core tree defined by Bavard and Walker \cite{BavWal} we give $S$ a convenient pants decomposition which will be used to construct a Riemann surface structure $R$ on $S$. This structure comes with an exhaustion sequence of Riemann sub-surfaces of finite type $\{R_n\}$.
The exhaustion $\{R_n\}$ has the crucial property: there exists an $N$ such that for $K\geq N$ and every $n \geq K $ all $K$-quasiconformal maps in Mod(R) leave $R_n$ invariant modulo free homotopy. Another way to say the latter is that at some point in the exhaustion the $K$-quasiconformal maps will leave invariant the sub-surfaces. 

The later property enables us to define a restriction map $\Psi_n:Mod(R)_n \to Mod(R_n)$ from the set of equivalence classes of $n$-quasiconformal homeomorphisms of $R$ to the countable group $Mod(R_n)$, which comprises equivalence classes of quasiconformal homeomorphisms of $R_n$ up to homotopy. We establish the countability of $Mod(R)$ by showing the injectivity of $\Psi_n$.

Chandran, Patel and Vlamis also mention a brief strategy that resembles ours to proof $Mod(R)$ is equal to the subgroup of $Map(S)$ generated by the compactly supported mapping classes. There is a slight difference from what we do, since we are constructing a pair of pants decomposition of $S$ by core trees. Since there is an element of choice, it is possible there exist conformal maps in $Mod(R)$ that induce homeomorphisms on the ends of $S$ which are not the identity. An example of the latter could be a rotation. Hence $Mod(R)$ may not be equal to the group of compactly supported mapping classes.

\section*{Acknowledgements}

We thank Ferrán Valdez for his invaluable feedback that improved this manuscript. To Tommaso Cremaschi and Yassin Chandran for useful discussions. Thanks also to Carlos Matheus and the CMLS, École Polytechnique for providing the facilities where this research was done. Finally many thanks to the Solomon Lefschetz Laboratory at UNAM, CONACYT Ph.d. grant, and UNAM PAPIIT IN-101422 for providing funding during the visit at L'X. We thank the anonymous referee for the helpful comments that improved the exposition of the paper. 

\section{Preliminaries}

As previously mentioned, for the construction of the Riemann surface $R$ we use a convenient pants decomposition of $S$. This is achieved using the notion of $\emph{core tree}$, denoted as $T=T(S)$, which we explain in what follows. 

\subsection{Core trees}

For a detailed construction we refer to Lemma 2.3.1 of \cite{BavWal}. The construction applies in fact to any connected, orientable infinite type surface. We recall here briefly the properties of a core tree:

\begin{itemize}
    
    \item It is an rooted infinite tree where each vertex can only have degrees $1,2$ or $3$.
    
    \item The vertices are either marked or unmarked. The unmarked only have degree $1$ or $3$. The root if unmarked is of degree $3$.
    
    \item The space $Ends(T)$ is homeomorphic to $Ends(S)$. Here $Ends(\cdot)$ is the space of ends of a topological space. See \cite{Richards} for an exposition on ends. 
    
\end{itemize}

\begin{figure}[!htb]
\hspace{0.05\textwidth}
      \begin{subfigure}{6cm}
   \includegraphics[width=6cm]{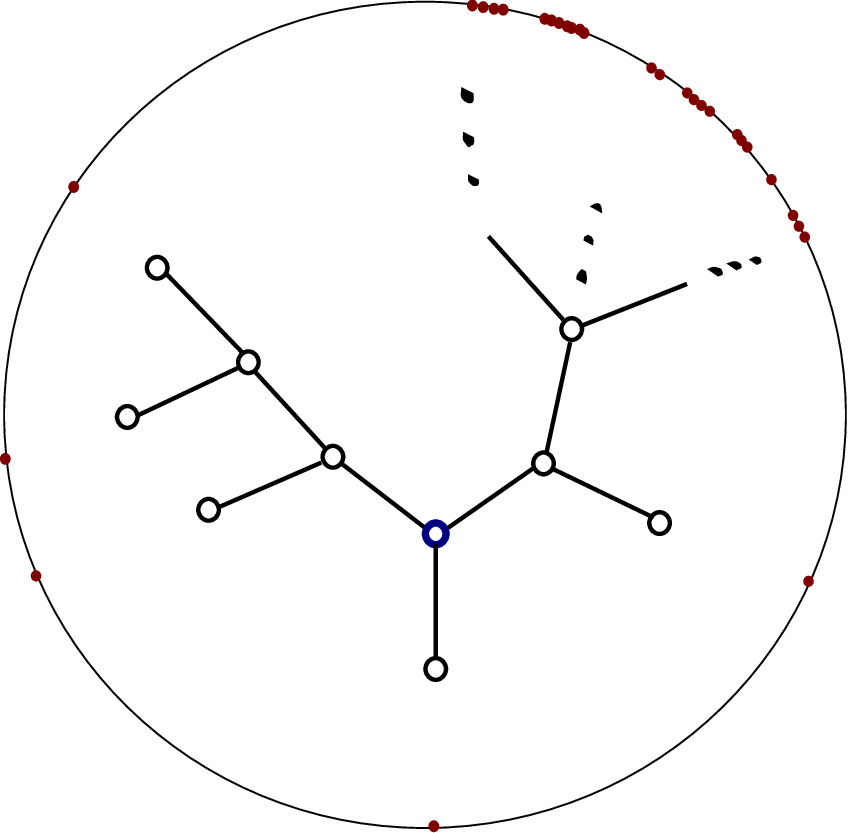}
  \newsubcap{Core tree. The root is blue colored and the ends are represented on the circle.}\label{tree}
  \end{subfigure}
  \hspace{0.05\textwidth}
\begin{subfigure}{6cm}
  \includegraphics[width=6cm]{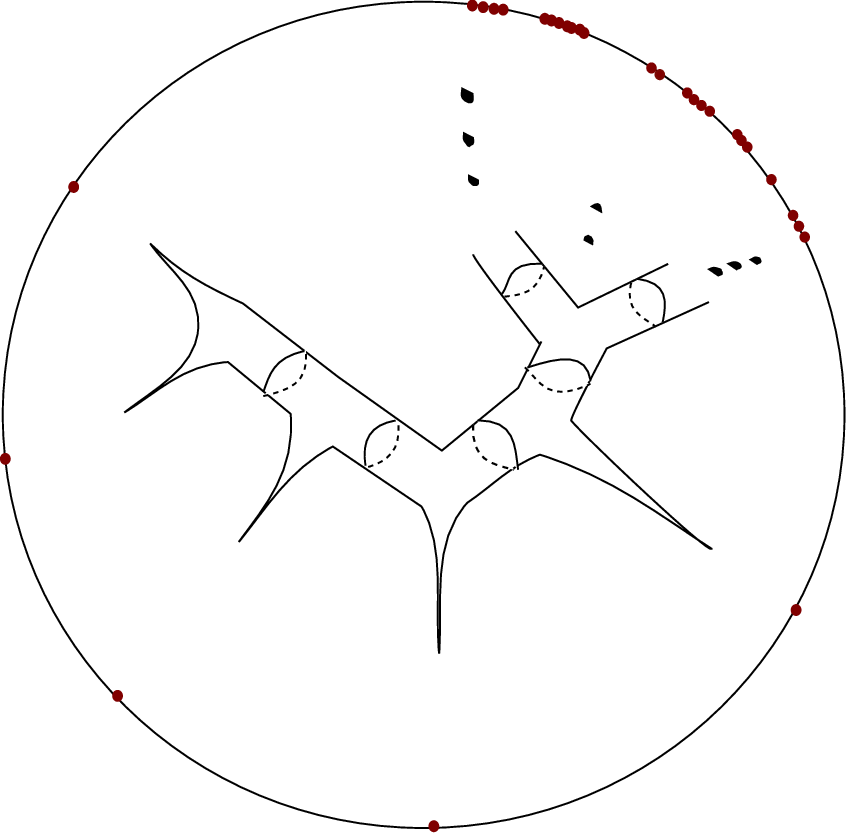}
  \newsubcap{Pair of pants decomposition}\label{fulltree}
\end{subfigure}
    \end{figure}

 See Fig. \ref{tree} for an example of a core tree.  The core tree serves as a scaffolding to construct the desired pairs of pants decomposition on a surface $S(T)$ homeomorphic to $S$. We detail its contruction in what follows. This is essentially section $2.3$ of \cite{BavWal}.

\begin{itemize}
    \item \textbf{Replacing vertices}. Let $v \in V(T(S))$ be unmarked. If $v$ has degree $3$ then replace $v$ with a pairs of pants. If $v$ is a leaf then just remove it. 
    
    If $v$ is marked then replace it with a torus with $deg(v)$ boundary components. Each torus has its own pairs of pants decomposition as in the Figure \ref{tori}.

\begin{figure}[!htb]
\centering
       \minipage{0.7\textwidth}
   \includegraphics[width=\linewidth]{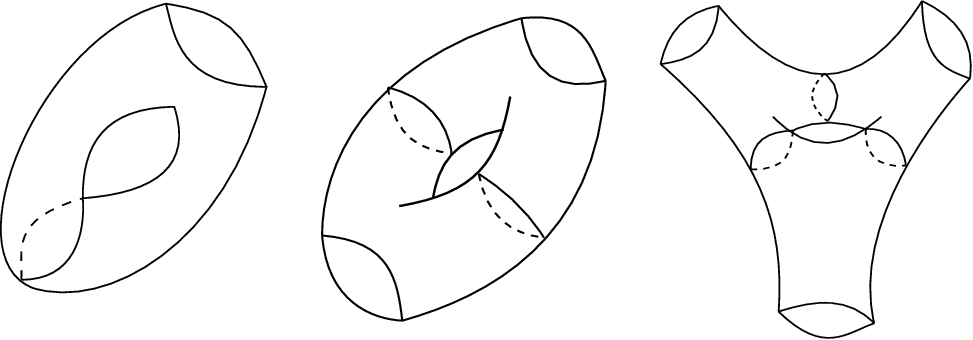}
  \caption{Tori with boundaries and with a pants decomposition. From left to right, we have the torus with one, two or three boundary components. }\label{tori}
   \endminipage
    \end{figure}
    
\item \textbf{Glueings}. 
 We glue pairs of boundary components the tori or pairs of pants if there was an edge connecting the replaced vertices by a zero twist. 

\item \textbf{Removing boundaries} Once glued we have a surface $S'$. Since $S$ has no boundary replace the boundary components of $S'$ by punctures. Now this new surface is $S(T)$ which we denote also as $S$ since it is homeomorphic to $S$.

\end{itemize}

In this pants decomposition of $S$, we have pairs of pants which can have $0,1$ or $2$ punctures. We refer to them as pairs of pants of type $0$,$1$ and $2$, respectively. See Fig. \ref{stepscoretree} for an image of the steps described above in a toy example. See algo Fig. \ref{fulltree} for an example of genus $0$ and Figs. \ref{treegenus} and \ref{fulltreegenus} for positive genus.

\begin{figure}[!htb]
\centering
       \minipage{0.9\textwidth}
   \includegraphics[width=\linewidth]{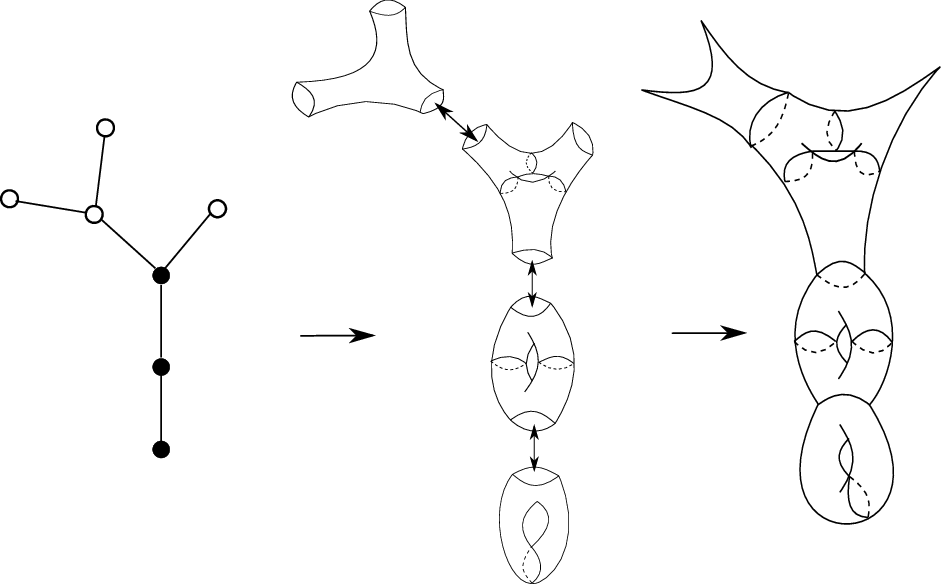}
  \caption{From a tree to a surface.}\label{stepscoretree}
   \endminipage
    \end{figure}

\begin{remark}
Every surface $S$ of infinite type is accompanied by a non-unique core tree, since there is an element of choice.
\end{remark}

\subsubsection{Preferred core tree}

For technical reasons that will be discussed later, it is convenient to work with a particular class of core trees. These are core trees having at most one \emph{exterior tree}. An \textbf{exterior tree} is a sub-tree comprising three unmarked vertices and two leaves. The exterior trees corresponds to pairs of pants of type $2$. As the following Lemma shows, for the surfaces that interest us, we can always find such a core tree. 

\begin{Lemma}
\label{exteriortree}
Let $S$ be of infinite type with $g=0$ and $End(S)$ having infinitely many isolated points. Then there is a core tree, $T(S)$ with at most one exterior tree. 
\end{Lemma}

\begin{proof}

Let $T(S)$ be a core tree of $S$. Since $T(S)$ has a countable number of vertices there is at most a countable number of exterior trees. Then there is a enumeration of them, denoted by $\{t_i\}_{i \in \mathbb{N}}$. We will remove them so that after the removal the new core tree $T'(S)$ has the property that $S(T')$ is homeomorphic to $S$ with at most one exterior tree. Let $p_i$ denote the geodesic path from the vertex of degree 2 of $t_i$ to the root. For geodesic path we mean a path  (as in graph theory) without repeating edges. Here $p_i(j)$ denotes the vertices along the path with $p_i(0)$ the initial vertex and $p_i(m_i)$ the root, $j \in \mathbb{N}$. If $p_i(0)$ is the root, we do not remove the exterior tree. Denote by $T(p_i(j))$ the connected component of $T(S) \setminus p_i(j)$ which contains the geodesic segment $[p_i(0),p_i(j-1)]$, for $0 < j \leq m_i$. Since the geodesic path $p_i$ is finite there is a unique maximal finite sub-tree $T(p_i(l))$ which has $t_i$ as the solely numbered exterior tree. Replace $T(p_i(l))$ with one single leaf, as in Fig. \ref{imagepreftree}. Note that $p_i(l)$ does not change its degree. If $p_i(l) \neq p_i(m_i)$ then $p_i(l)$ is not the vertex of a new exterior tree otherwise $T(p_i(l))$ would not have been maximal. The only possibility after this procedure is that a new exterior tree occurs on the root. Finally, since for every $t_i$ we replaced finitely many leaves with one single leaf and $Ends(S)$ has infinitely many isolated points, we have that $Ends(T')$ is homeomorphic to $Ends(T)$. Hence, $S(T')$ is homeomorphic to $S(T)$.

\end{proof}

\begin{figure}[!htb]
\centering
       \minipage{0.9\textwidth}
   \includegraphics[width=\linewidth]{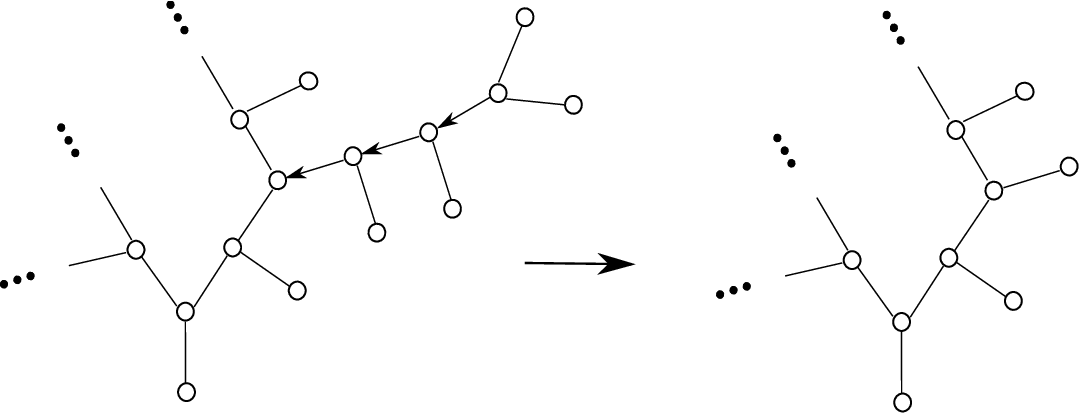}
  \caption{Removing a maximal sub-tree. The path made from directed edges represent a maximal sub-tree.}\label{imagepreftree}
   \endminipage
    \end{figure}

\subsection{Hyperbolic structures of the first kind on $S$}\label{sechypfirstkind}

By prescribing lengths to the closed curves of a pants decomposition (for example given by a core tree) of $S$ we obtain a hyperbolic structure on it, say $R'$, which it may not be complete. This is because each pair of pants of type $0,1$ and $2$ becomes a hyperbolic pairs of pants with cusps, and there is the possibility that some geodesic reach some end at finite time (see \cite{Matsuzaki}). Each hyperbolic pair of pants with $0,1$ or $2$ cusps will be called hyperbolic pairs of pants of type $0,1$ and $2$, respectively. Following \cite{completehypmanifolds} we call a connected component of the geodesic boundary of a hyperbolic pairs of pants a \emph{cuff}.

According to \cite{completehypmanifolds} a geodesic $\gamma$ of $R$ exits an end $e$ if it converges to it. An end $e$ is visible if there exists an open subset $V$ of the unit tangent bundle on $R$ such that for any $v \in V$, the geodesic starting from $v$ exits $e$. A non-visible end is an end which is not visible. A sequence of pair of pants exits an end $e$ if it converges to $e$. An infinite type end $e$ is an end which is not isolated. With the previous terminology we have.

\begin{proposition}(Consequence of Proposition 3.7 (ii) and (iii) of \cite{completehypmanifolds})

Let $R$ be a complete hyperbolic surface with a hyperbolic pairs of pants decomposition, then
 \begin{enumerate}
     \item an infinite type end $e$ of $R$ is non-visible if and only if for any geodesic ray $\gamma$ that exits $e$ the sequence of pants in the decomposition of $R$ that $\gamma$ intersects also exit $e$,
     
     \item $R$ is of the first kind if and only if each end of $R$ is non-visible.
 \end{enumerate}

\label{firstkindprop}
\end{proposition}

\begin{theorem} (Theorem 5.1 of \cite{completehypmanifolds})
Let $R'$ be a (not necessarily complete) hyperbolic surface with a pants decomposition. Then there exists a choice of twists along the cuffs of the pants so that the induced hyperbolic surface $R$, after possibly adding funnels, is a geodesically complete hyperbolic surface.
\label{completeness}
\end{theorem}

\begin{proposition}
Any connected, orientable infinite type surface without boundary $S$ admits a complete hyperbolic structure of the first kind. 
\label{everyinffirstkind}
\end{proposition}

\begin{proof}
Let $S$ have a pants decomposition given by a core tree. Assign any lengths to each of the closed curves of the pants decomposition and obtain a (possibly incomplete) hyperbolic structure $R'$ on $S$. Because of Theorem \ref{completeness} there is choice of twists on the cuffs of the pants decomposition of $R'$ that induce a complete hyperbolic structure $R$ on $S$. Since the pants decomposition is given by a core tree Proposition \ref{firstkindprop} gives that any infinite type end of $R$ is non-visible. A cusp is an end which is non-visible. Therefore $R$ is of the first-kind. 
\end{proof}

\begin{remark}
The previous result is implicit in the work of Basmajian and Šarić \cite{completehypmanifolds}, but it is not stated. It is not necessary the pants decomposition given by a core tree for the Proposition \ref{everyinffirstkind} to be true. We just think that with a core tree it is easy to see a non-visible end. 
\end{remark}

As explained in the proof of Proposition 1 of \cite{Matsuzaki} it is necessary for $S$ to have a hyperbolic structure of the first kind for the possibility of $Mod(R)$ to be countable. The above construction is still not enough to show $Mod(R)$ is countable. The obstruction are the lengths of the cuffs, again because of  Proposition 1 of \cite{Matsuzaki}. We denote by $\mathcal{P}(R)$ the hyperbolic pants decomposition on $R$ induced by the corresponding core tree of $S$. 

\section{Proof of the main result.}

The proof begins in this section, which is divided in two parts. In the first part we construct a complete and of the first kind Riemann surface structure $R$ on $S$. Here the pants decomposition given by a core tree is convenient since we can assign by \emph{levels} the length of the closed curves of the decomposition. In the second part we show countability of the group $Mod(R)$.

\subsection{Constructing the Riemann surface structure}

\textbf{Case 1. $End(S)$ does not have a finite number of isolated points, and genus of $S$ is $0$.}

        By Lemma \ref{exteriortree} we choose $T(S)$ having at most one exterior tree. We will only label edges of $T(S)$ that are not incident to a leaf. Choose an edge $E_1$ of the root vertex that does not meet a leaf and label it $1!$. We assign the number $3!$ to the other incident edges of the endpoints $E_1$ (as long as they do not meet a leaf). For the edges that got a number we repeat the process as with $E_1$. Their remaining unlabelled neighboring edges will have the number $5!$. We proceed by induction so that the edges are assigned the factorials of the odd numbers. See Fig \ref{redtreegenus0} for an example.

\begin{figure}[!htb]
\hspace{0.05\textwidth}
      \begin{subfigure}{6cm}
   \includegraphics[width=6cm]{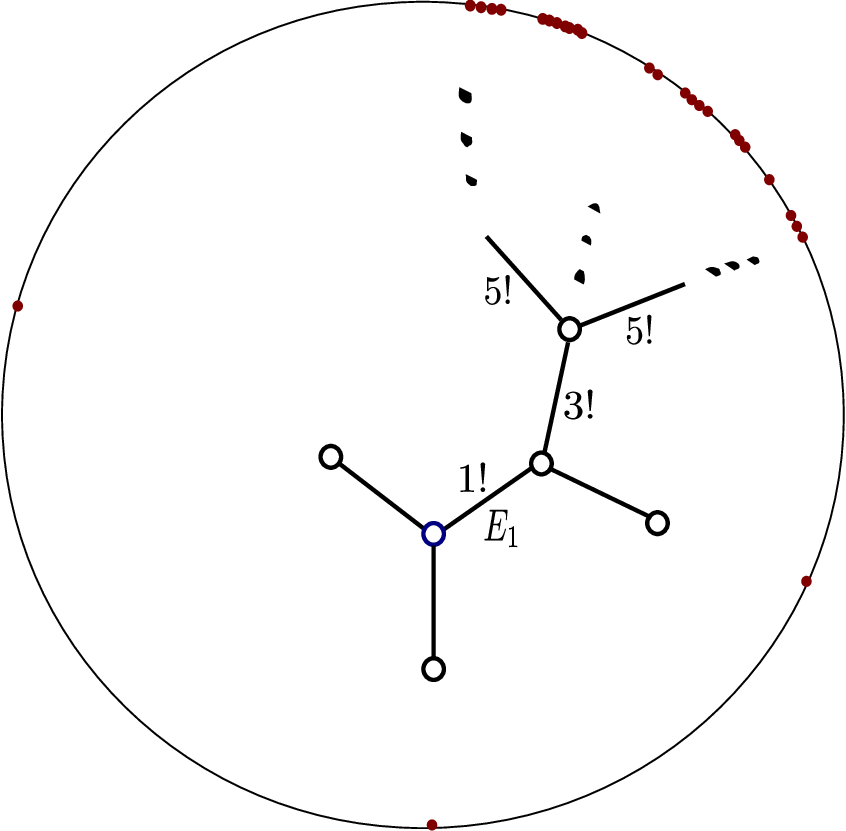}
  \newsubcap{Core tree with at most one exterior tree. The root is in blue. }\label{redtreegenus0}
   \end{subfigure}
   \hspace{0.05\textwidth}
\begin{subfigure}{6cm}
  \includegraphics[width=6cm]{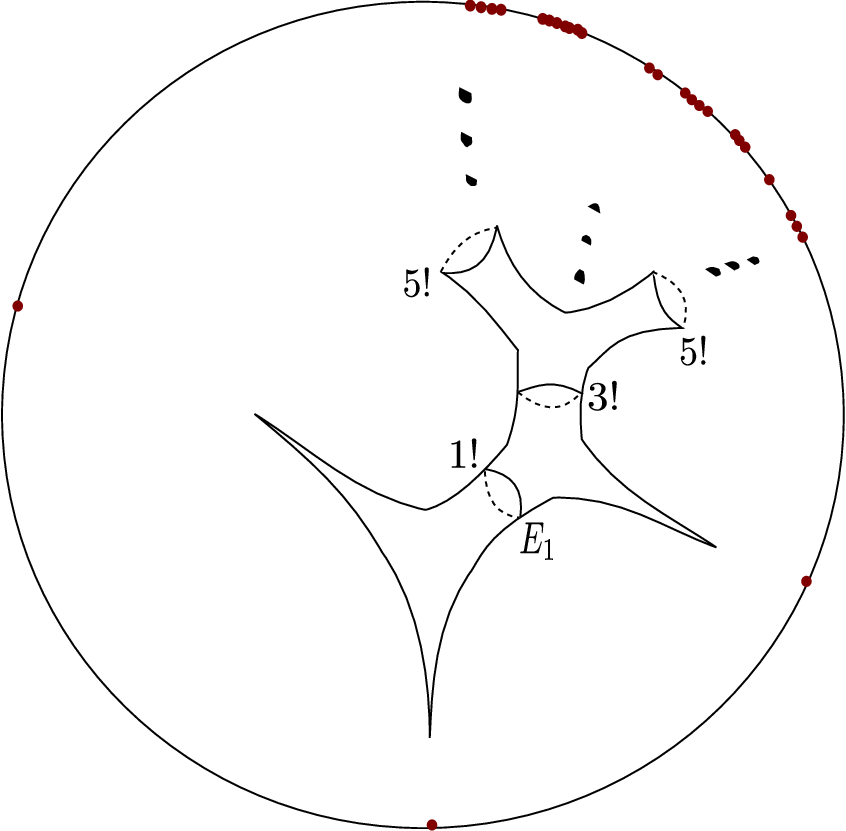}
  \newsubcap{Surface with at most one hyperbolic pairs of pants of type 2.}\label{redfulltree}
\end{subfigure}
    \end{figure}
    
    We obtain from $T(S)$ a pants decomposition on $S$ as in the preceding section. The cuffs of the pairs of pants decomposition will have the hyperbolic length equal to the number assigned to the corresponding edge in $T(S)$. The edges that do not meet a leaf always correspond to a cuff. Hence, we have a Riemann surface structure $R'$ exhausted by surfaces of finite topological type $\{R_n\}_{n \in \mathbb{N}}$. We describe the exhaustion in what follows. The finite subsurface $R_1$ is obtained by gluing two pairs of pants along $E_1$; each of its boundary components has length $3!$. $R_2$ is the union of $R_1$ with the pairs of pants glued along the boundary components of $R_1$. We proceed by induction, and therefore we have sub-surfaces $R_n$ with geodesic frontier whom all components have the same length of $(2n+1)!$. See Fig \ref{redfulltree}.
        
     By the proof of Proposition \ref{everyinffirstkind} we obtain a complete hyperbolic surface of the first kind $R$ from $R'$.

\textbf{Case 2. Genus of $S$ is positive or has a finite number of isolated ends }

        Note that we can always choose a core tree $T(S)$ with a finite sub-tree $T^*$ with all of its  vertices being marked and also having all possibly finite leaves. We can choose the tree $T(S)$ such that the root vertex is part of $T^{*}$ and deleting it will separate $T^{*}$ from the rest of $T(S)$. We can also choose $T(S)$ such that $T(S) \setminus T^{*}$ has at most one exterior tree (in the case of infinitely many isolated points). So we choose $T(S)$ in this way.
       
       Assign whatever positive numbers to the edges of $T^{*}$. Then choose an edge $E_1$ of the root vertex that is not part of $T^{*}$ and then assign numbers to the edges as in the previous case. Obtain from $T(S)$ the pants decomposition as in section $2$ and make the twists to each cuff such that the resulting Riemann surface $R$ is complete and of the first kind. See Figs. \ref{treegenus} and \ref{fulltreegenus}.

\begin{figure}[!htb]
\hspace{0.05\textwidth}
    \begin{subfigure}{6cm}
   \includegraphics[width=6cm]{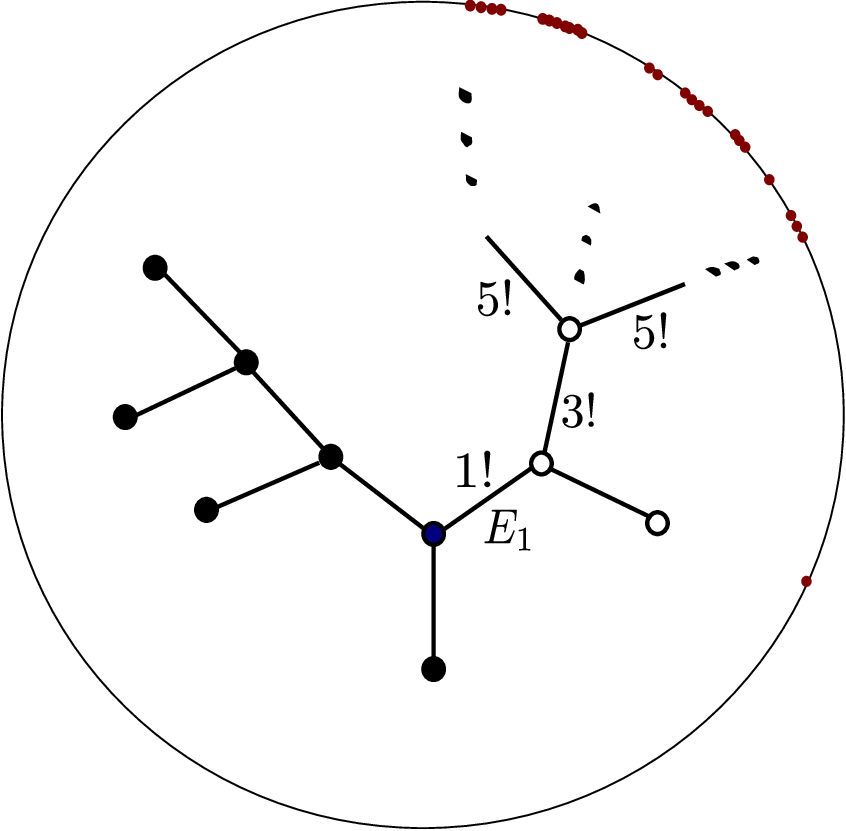}
  \newsubcap{Core tree representing finite genus. The root is in blue. }\label{treegenus}
 \end{subfigure}
 \hspace{0.05\textwidth}
\begin{subfigure}{6cm}
  \includegraphics[width=6cm]{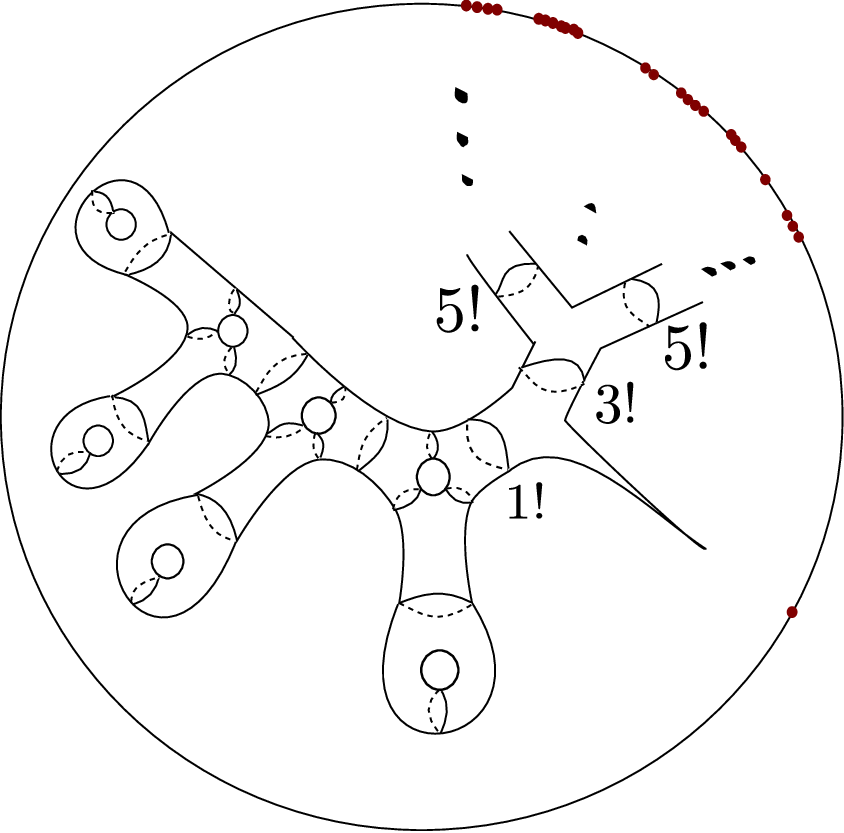}
  \newsubcap{Surface with positive genus.}\label{fulltreegenus}
\end{subfigure}
    \end{figure}

       The exhaustion of the $R$
      will be given by: $R_1$ is the subsurface with all the genus of $S$ and possibly the finite isolated points of $Ends(S)$. The geodesic boundary of $R_1$ is the cuff chosen corresponding to $E_1$. The subsurface $R_2$ is $R_1$ union the pair of pants of type $1$ or pair of pants of type $0$ that $E_1$ bounds. The subsurface $R_3$ is the union of $R_2$ and the perifec pair or pair of pants the geodesic boundaries of $R_2$ bounds. We proceed by induction.

\subsection{Countability}

As mentioned in the introduction, the idea for proving countability is to define injective maps $\Psi_n: Mod(R)_n \rightarrow Mod(R_n)$ for all $n \geq N$, for some $N$. Recall that $Mod(R)_n$ is the set of quasiconformal equivalence classes with a $n$-quasiconformal representative; $Mod(R_n)$ is the group of quasiconformal maps of $R_n$ up to homotopy. The injectivity of $\Psi_n$ implies that $Mod(R)$ is countable since $Mod(R) = \bigcup_{n \in \mathbb{N}} Mod(R)_n$ and $Mod(R)_m \subset Mod(R)_n, n>m$. This $N$ is given by Lemma \ref{nonwanderingsurface}, which is an extension of Lemma 4 of \cite{Matsuzaki}. In order to prove it we need the following two lemmas. 

\begin{Lemma}(Wolpert's Lemma, \cite{Wolpert})
\label{wolpertLemma}
Let $c$ be a simple closed geodesic on a Riemann surface $R$ with the length $l(c)$ and $g: R \rightarrow R$ a $K$-quasiconformal homeomorphism. Then the geodesic length $l(g(c))$ of the free homotopy class of $g(c)$ on $R$ satisfies 

$$ \dfrac{1}{K}l(c) \leq l(g(c)) \leq K \cdot  l(c).$$
\end{Lemma}
\vspace{0.5cm}

This is a consequence of Lemma 8.1 of \cite{completehypmanifolds}. 

\begin{Lemma}
\label{boundpenthagon}
Let $\Sigma$ be an ideal right-angled pentagon as in Fig. \ref{idealpenthagon}. Let $l(d)$ the length of the geodesic orthogonal segment $d$. Then for sufficiently large $l(a)>1$, and $l(c) > l(a) $, then 

$$l(d) \geq  l(c)-l(a). $$ 
\end{Lemma}

\begin{figure}[!htb]
\centering
       \minipage{0.4\textwidth}
   \includegraphics[width=\linewidth]{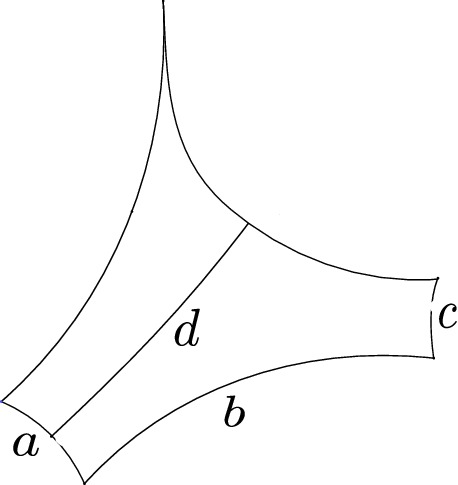}
  \caption{Ideal Pentagon}\label{idealpenthagon}
   \endminipage
    \end{figure}

  \begin{Lemma}
\label{nonwanderingsurface}
Let $S$ be of infinite type and finite genus and $R$ be a Riemann surface structure on $S$ as constructed in Section 3.1. There exists $N \in \mathbb{N}$ such that for all $g:R \rightarrow R $ $K$-quasiconformal homeomorphism, with $K \geq N$, $g(R_n)$ is freely homotopic to $R_n$ for every $n \geq K$. The $R_n$ are the exhaustion subsurfaces mentioned in Section 3.1. 
\end{Lemma}

\begin{proof}

The goal will be to contradict Lemma \ref{wolpertLemma}. The closed curves we are going to use are the geodesic boundaries of the subsurfaces $R_n$, which are closed curves of $\mathcal{P}(R)$ described in Section 3.1. Lemma \ref{boundpenthagon} will give us the $N$ such that if $R_n, n\geq N,$ is not freely homotopic to $g(R_n)$, then the image of some of the boundary curves, say $c$ under $g$ is going to be large enough so that the inequality $l(g(c)) \leq K \cdot l(c) $ of Lemma \ref{wolpertLemma} cannot hold. This is where Lemma \ref{exteriortree} is crucial. If we allow infinitely many pants of type $2$, then by Theorem 7.17.1 (i) of \cite{Beardon}, the orthogonal geodesic that goes from the geodesic boundary to the geodesic boundary of the pants will go to zero. Hence, we cannot longer control de length of the image of a geodesic boundary by a quasiconformal map.

\textbf{Case 1. $S$ is of genus $0$ and $Ends(S)$ has infinitely many isolated points.}

First note that for every $n$ if $g(R_n)$ is not freely homotopic to $R_n$ there is some curve $c_n$ of the boundary component of the $R_n$ such that $g(c_n) \cap (R \setminus \overbar{R_n}) \neq \emptyset$. This is because the boundary components of $R_n$ are elements of $\mathcal{P}(R)$ described in Section 2 and the $R_n$'s are sub-surfaces with boundaries. Recall from Section 2 that we eliminated the possibility of having infinitely many pairs of pants of type $2$, in fact at most one can occur only in the root. Hence, the $\mathcal{P}(R)$ consists of, at most, one hyperbolic pair of type $2$, and infinitely many of pairs of pants of type $0$ or $1$. Suppose then we have $g(c_n) \cap (R \setminus \overbar{R_n}) \neq \emptyset$, for some $c_n$. Since we have a direction of exhaustion, that is, the exhaustion $\{R_n\}$ is a nested sequence: $R_n \subset R_{n+1}$,  and we have genus $0$, there is some geodesic $\sigma_n$ homotopic to $g(c_n)$ such that it has a segment $s_n$ which cuts a hyperbolic pair of pants of type $0$ or $1$. The segment $s_n$ cuts the pairs of pants in a boundary component and returns to it in the direction of the orthogonal geodesic on the corresponding pairs of pants that also cuts the same boundary component, see Figures \ref{segment1} and \ref{segment2}. We are going to give lower bounds to the length of this segment. 

\begin{figure}[!htb]
\hspace{0.05\textwidth}
    \begin{subfigure}{6cm}
   \includegraphics[width=5cm]{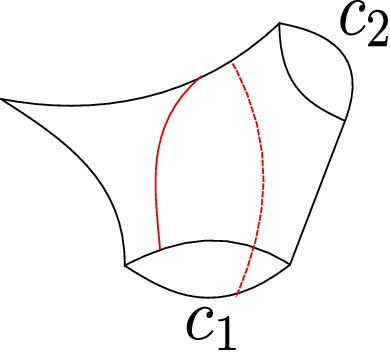}
  \newsubcap{In red is a possible trajectory of the segment $s_n$ in the case of pairs of pants of type $1$. The lengths of $c_1$ is $(2m+1)!$ and of $c_2$ is $(2m+3)!$, for some $m$. }\label{segment1}
 \end{subfigure}
 \hspace{0.05\textwidth}
\begin{subfigure}{6cm}
  \includegraphics[width=4.5cm]{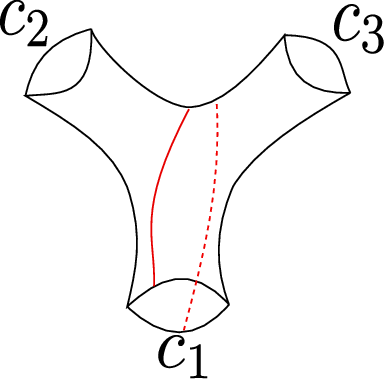}
  \newsubcap{In red is a possible trajectory of the segment $s_n$  in the case of pairs of pants of type $0$. The lengths of $c_1$ is $(2m+1)!$ and of $c_2$ and $c_3$ are $(2m+3)!$, for some $m$.}\label{segment2}
\end{subfigure}
    \end{figure}

Suppose first $\mathcal{P}(R)$ has infinitely many pairs of pants of type $1$. Then Lemma \ref{boundpenthagon} gives an $N$, such that if $s_n$ cuts a pair of pants of type $1$ with one geodesic boundary of length $(2m+1)!$ and the other of length $(2m+3)!$, with $m \geq N$ then the inequality given by Lemma \ref{boundpenthagon} implies $l(s_n) > (2m+1)!(2m+1)$. If $s_n$ would cut a pairs of pants of type $0$ on the boundary component of length $(2m+1)!$ (the other two boundary components have length $(2m+3)!$), with $m\geq N$, then the calculations for the proof of  Proposition $5$ of \cite{Matsuzaki} also implies $l(s_n) > (2m+1)!(2m+1)$. Therefore in this case we can take $N=l(a)$ of Lemma \ref{boundpenthagon} for the statement we want to prove. Hence if $g$  is a $K$-quasiconformal map with $K\geq N$ with $g(R_n)$ not freely homotopic to $R_n$, then there would be a pair of pants such that $s_n$ cuts its boundary component of length $(2m+1)!$ with $m\geq n \geq K \geq N$. The previous discussion shows together with Lemma \ref{wolpertLemma} that 
$n(2n+1)! \geq l(g(c_n))$, but $l(g(c_n)) \geq l(s_n) > (2m+1)!(2m+1)$, contradiction.  Here we have used the assignment of lengths by the factorials of odd numbers for easily contradict Wolpert's lemma. Other assignments of lengths might be possible. 

In the case where $\mathcal{P}(R)$ has $r \in \mathbb{N}$ pairs of pants of type $1$ $P_{i}^1$, $i \in {1,2,...,r}$, we set $N=min\{n: P_i \subset R_n  \forall i \in {1,2,...,r}\}+1$. The same argument in the case of pairs of pants of type $0$ gives the result. 

\textbf{Case 2. $S$ has positive genus and possibly finite isolated ends.}

We can choose $N$ as in the previous case. Note however that given our construction of this case in section 2 the subsurface with genus or finite isolated ends is a non wandering surface (a sub-surface fixed by homeomorphisms up to homotopy). Therefore all the arguments go through as in the previous case. 
  
\end{proof}

\begin{proof}[End of proof of Theorem \ref{MainResult}]
By Lemma \ref{nonwanderingsurface} the restriction map $\Psi_n: Mod(R)_{n} \rightarrow Mod(R_n)$ is well defined for $n \geq N$. Recall that $Mod(R)_n$ are the homotopy classes of $Map(S)$ that have a $n$-quasiconformal representative for the Riemann surface structure we constructed. The group $Mod(R_n)$ is the group of homotopy classes of quasiconformal homeomorphisms of $R_n$. Note that because of the decomposition given by a core tree we have that the connected components of $R_{n+1} \setminus \overbar{R_n}$ are pairs of pants. Because of Lemma \ref{nonwanderingsurface} for every $n$-quasiconformal homeomorphism $g$, with $n \geq N$ we have $g(R_{n+1} \setminus \overbar{R_n} ) = R_{n+1} \setminus \overbar{R_n} $ mapping any pairs of pants (of type $0$ or $1$) to a pair of pants of the same type. The homeomorphism $g$ also maps boundary components to boundary components of the same length. This implies that if $\Psi_n(g)=\Psi_n(f)$, for $m \geq n$ the difference between $g$ and $f$, i.e. $g \circ f^{-1}$, on $R_{m+1} \setminus \overbar{R_m}$ is a conformal homeomorphism homotopic to the identity. Hence the difference between $f$ and $g$ on $R \setminus \overbar{R_n}$ is by possibly Dehn twists along each cuff of $R \setminus \overbar{R_n}$, but this cannot happen, otherwise the Dehn twists would make either $f$ or $g$ $K$-quasiconformal with $K>n$. This follows from Matsuzaki's argument for proving Theorem 3 in  \cite{Matsuzaki}, which essentially goes as follows. The quasiconformal constant $K$ of extremal (in)finite products of Dehn twists for a complete hyperbolic structure (Theorem 1 of \cite{Dehntwistsproducts}) is bounded: $$K \geq \sqrt{\Big(\dfrac{1+(2n+1)!)}{\pi}\Big)^2 + 1}.$$ But $K \leq n$, a contradiction. Then $Mod(R)_{n}$ is countable.
\end{proof}

\begin{note}
When infinite genus is involved the above method has a main limitation, as one might consider non-separating curves. The difficulty then is to show that a $n$-quasiconformal map $g$ (as in the proof of Theorem \ref{MainResult}) has the property $g(R_{n+1} \setminus \overbar{R_n} ) = R_{n+1} \setminus \overbar{R_n} $. The map $g$ may act as a Dehn twist around the genus, hence it may not map the boundaries of $R_n$ to themselves. If we do not consider non-separating curves, now the issue is to verify if $g|_{(R_{n+1} \setminus \overbar{R_n}} $ is conformal or maybe quasiconformal but with restricted distortion. 
\end{note}

\printbibliography

\Address
\end{document}